\documentclass{amsart}

\usepackage{amssymb}

\usepackage{graphicx}

\usepackage{color}

\usepackage{tikz}
\usepackage{tikzscale}
\usepackage{pgfplots}
\pgfplotsset{compat=newest}
\usetikzlibrary{arrows,arrows.meta}

\usetikzlibrary{calc}

\newcommand{\NN}{\mathbb{N}}
\newcommand{\RR}{\mathbb{R}}
\newcommand{\CC}{\mathbb{C}}

\newcommand{\ltwor}{L_r^2(Q)}
\newcommand{\ltworr}{L_{1/r}^2(Q)}
\newcommand{\htilde}{\tilde{H}^1(Q)}
\newcommand{\htildeo}{\tilde{H}_0^1(Q)}

\newcommand{\interp}{\Pi}
\newcommand{\mA}{\boldsymbol{\mathsf{A}}}
\newcommand{\mB}{\boldsymbol{\mathsf{B}}}
\newcommand{\mE}{\boldsymbol{\mathsf{E}}}
\newcommand{\bF}{\boldsymbol{F}}

\newcommand{\bu}{\boldsymbol{u}}
\newcommand{\bv}{\boldsymbol{v}}
\newcommand{\bq}{\boldsymbol{q}}

\newcommand{\bphi}{\boldsymbol{\phi}}

\newcommand{\bpsi}{\boldsymbol{\psi}}

\newcommand{\norm}[1]{\left\lVert#1\right\rVert}
\DeclareMathOperator*{\R}{Re}

\newtheorem{theorem}{Theorem}
\newtheorem{lemma}[theorem]{Lemma}

\newtheorem{remark}[theorem]{Remark}
\newtheorem{definition}[theorem]{Definition}

\begin{document}

\title[Convergence analysis of SBFEM]{Convergence analysis of the scaled
boundary finite element method for the Laplace equation}

\thanks{The first author gratefully acknowledges support by the German Research Foundation (DFG) in the Priority Programme SPP 1748 Reliable simulation techniques in solid mechanics under grant number BE6511/1-1. The second author is a member of the INdAM Research group GNCS and his research is partially supported by IMATI/CNR and by PRIN/MIUR. The first and third author gratefully acknowledge the support by Mercator Research Center Ruhr (MERCUR) under grant Pr-2017-0017. We would also like to thank the project partners Prof. Carolin Birk (Universit\"at Duisburg-Essen, Germany) and Prof. Christian Meyer (TU Dortmund, Germany) as well as Professor Gerhard Starke for the fruitful discussions.}

\author{Fleurianne Bertrand }
\address{Humboldt-Universit\"at zu Berlin, Germany, King Abdullah
University of Science and Technology, Saudi Arabia and University of Twente, Enschede, the Netherlands}
\email{fb@math.hu-berlin.de}

\author{Daniele Boffi}
\address{King Abdullah University of Science and Technology (KAUST), Saudi
Arabia and University of Pavia, Italy}

\author{Gonzalo G. de Diego}
\address{Mathematical Institute, University of Oxford, UK}
\email{gonzalezdedi@maths.ox.ac.uk}

\begin{abstract}
	The scaled boundary finite element method (SBFEM) is a relatively recent boundary element method that allows the approximation of solutions to PDEs without the need of a fundamental solution. A theoretical framework for the convergence analysis of SBFEM is proposed here. This is achieved by defining a space of semi-discrete functions and constructing an interpolation operator onto this space. We prove error estimates for this interpolation operator and show that optimal convergence to the solution can be obtained in SBFEM. These theoretical results are backed by a numerical example.
\end{abstract}

\maketitle

\section{Introduction}

The scaled boundary finite element method (SBFEM), first proposed by Song and
Wolf, is a boundary element method that does not require a fundamental
solution. It has proven to be particularly effective for problems with
singularities or posed over unbounded media, see~
\cite{song1997,song1999,wolf2003}. In SBFEM, a semi-analytical (or semi-discrete, as we shall call it
here) solution to a PDE is constructed by transforming the weak formulation of
the PDE into an ODE. Essentially, given a star-shaped domain $\Omega\subset
\RR^n$, a coordinate transformation is performed (the scaled boundary
transformation) in terms of a radial variable and $n-1$ circumferential
variables. Then, an approximate solution is sought in a space of functions
discretized solely in the circumferential direction. The resulting weak
formulation posed over this space is then transformed into an ODE which, under
certain circumstances, can be solved exactly, yielding a semi-analytical
approximation of the solution to the PDE.

The SBFEM has been applied to a wide range of problems that arise in science
and engineering, such as crack propagation~\cite{song2003} and
acoustic-structure interactions~\cite{liu2019}. Moreover, the limitation to
star-shaped domains has been overcome with the development of scaled-boundary
polygon elements, in which the domain is broken into arbitrarily shaped
polygons and shape functions are constructed over these polygons based on
SBFEM~\cite{Ooi2012,Ooi2014,gravenkamp2016}.

The objective of this paper is to introduce a rigorous framework in which the
error of the approximate solution obtained by SBFEM can be estimated. In
particular, the notion of a semi-discrete solution to a PDE is formalized by
defining a space of semi-discrete functions and constructing an interpolation
operator onto this space. Then, given a semi-analytical solution obtained in
the framework of SBFEM, estimates of its error can be obtained by bounding the
interpolation operator's error using C\'ea's lemma. We limit the analysis to
Poisson's equation posed on a circular domain for simplicity; this setting is
appropriate to highlight the main features of our theoretical setting.

The overview of this paper is as follows: in Section~\ref{se:setting} we
describe the continuous problem together with the polar coordinate change of
variables. In Section~\ref{se:semi-discrete} we introduce a
semi-discretization of our problem, where the domain is discretized only in
the angular coordinate. It is shown that the semi-discrete solution converges
optimally to the continuous solution.  Section~\ref{se:sbfem}, making use of
the semidiscretization, transforms the original problem into an ODE.  Finally,
two numerical results reported in Section~\ref{se:num} show that the method
is performing optimally also in the presence of singularities.

\section{Setting of the problem}
\label{se:setting}

Given an angle $\Theta$ in $(0,2\pi)$, we are considering the Poisson problem
in the following circular sector (see Figure~\ref{fig:pacman_domain}):
\begin{equation*}
\Omega := \left\{ (x,y)\in\RR^2 : 0< x^2 + y^2 < 1,\
0<\arctan{\left( \frac{y}{x} \right)}<\Theta \right\}.
\end{equation*}
%
\begin{figure}
\centering
\begin{tikzpicture}[scale=2]
        \draw [dashed,-{Latex[scale=1]}] (0.35,0) arc [radius=0.35, start angle=0, end angle= 225];
        \draw (0,0.35) node [above left] {$\Theta$};
        \draw (1,0) arc [radius=1, start angle=0, end angle= 225];
        \draw (0,0) --(1,0);
        \draw (0,0) --(-0.70710,-0.70710);
\end{tikzpicture}
\caption{Sector of a disk of angle $\Theta$.}
\label{fig:pacman_domain}
\end{figure}
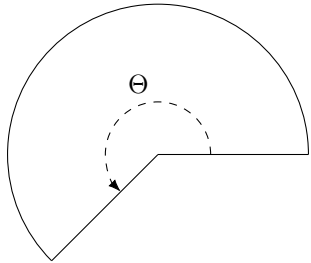
Since we are going to consider a change of variables when defining the scaled
boundary method, we denote with $\hat\bullet$ (with the hat symbol) quantities
defined on $\Omega$ that correspond to quantities $\bullet$ defined on the
reference domain.
Hence our problem reads: find $\hat{u}: \Omega\rightarrow\RR$ such that 
\begin{equation}
\label{eq:poisson}
\aligned
&-\hat{\Delta} \hat{u} = \hat{f} &&\text{in }\Omega\\
&\hat{u} = 0 && \text{on }\partial\Omega,
\endaligned
\end{equation}
where $\hat f\in L^2( \Omega)$ and $\hat{\Delta} = \partial^2_x +
\partial^2_y$ is the Laplace operator in the Cartesian coordinates $(x,y)$.

Let the curved part of the boundary of $\Omega$ be parametrized by the graph
\[
(x_b(\theta),y_b(\theta))=(\cos\theta,\sin\theta)\qquad
\theta\in(0,\Theta)
\]
and define the open rectangle $Q:=(0,1)\times (0,\Theta)$.
We consider the mapping $F:Q \rightarrow \RR^2$ given by
\begin{equation}
\label{eq:sb_map}
F(r,\theta) = r \left( x_b(\theta), y_b(\theta) \right).
\end{equation}

In this particular case, the \textbf{scaled boundary transformation} is given
by the change of variables $(r,\theta) = F^{-1}(x,y)$ for $(x,y)\in \Omega$,
i.e., by the polar coordinate transformation. The Jacobian of $F$ is given by
\[
DF(r,\theta) = \left( \begin{array}{cc}
        \partial_r x & \partial_r y \\
        \partial_\theta x & \partial_\theta y
\end{array} \right) = \left( \begin{array}{cc}
        \cos \theta  & \sin \theta \\
        -r\sin \theta & r \cos \theta
\end{array} \right)
\]
and its determinant is $|DF(r,\theta)| = r$. Since $F$ is differentiable and
$|DF(r,\theta)|$ is invertible in the open set $Q$ we have
\[
DF^{-1}(F(r,\theta)) = \left( \begin{array}{cc}
        \partial_{x} r & \partial_{x} \theta \\
        \partial_{y} r & \partial_{y} \theta
\end{array} \right)
 = \frac{1}{r} \left( \begin{array}{cc}
        r \cos\theta & -\sin\theta \\
        r \sin\theta & \cos\theta
\end{array} \right).
\]
Let $u(r,\theta) = \hat{u}(F(r,\theta))$, then the relation between the
gradient in Cartesian coordinates $\hat\nabla = (\partial_x, \partial_y)^\top$
and the gradient in polar coordinates
$\nabla = (\partial_r, \partial_\theta)^\top$ is given by
\[
\hat{\nabla} \hat{u}(x,y) =
DF^{-1}(x,y) \nabla u(F^{-1}(x,y)) \quad \text{in }\Omega.
\]
Moreover, the solution $\hat u$ of~\eqref{eq:poisson} satisfies
\begin{equation}
\label{eq:change_of_variables}
||\hat{u}||^2_{H^1(\Omega)}=
\int_0^1\int_0^{\Theta}\left(ru^2+r(\partial_r u)^2+
\frac{1}{r}(\partial_\theta u)^2\right)\,\mathrm{d}r\,\mathrm{d}\theta < \infty.
\end{equation}

In order to consider the variational formulation of~\eqref{eq:poisson} in
polar coordinates we have to consider appropriate weighted functional spaces.
While this is pretty straightforward and well understood, we present the
procedure in detail since the notation will be useful for the analysis of the
numerical approximation.

Given a weight function $w(r,\theta)$ in $Q$, we define the weighted Lebesgue
space
\[
L^2_{w}(Q) = \left\{v:Q\to\RR\text{ measurable}:
\int_0^1\int_0^{\Theta}v^2w\,\mathrm{d}r\,\mathrm{d}\theta < \infty\right\}
\]
with inner product
\[
(u,v)_{L^2_{w}(Q)}:=\int_0^1\int_0^{\Theta}uvw\,\mathrm{d}r\,\mathrm{d}\theta.
\]
We will use in particular $w=r$ and $w=1/r$; it is not difficult to see that
we have
$||u||_{L^2_{r}( Q)} \leq ||u||_{L^2( Q)} \leq ||u||_{L^2_{1/r}( Q)}$ for all
$u\in L^2_{1/r}( Q)$. Furthermore, these spaces are complete
\cite{kufner1985}.

The bound~\eqref{eq:change_of_variables} motivates the definition of the
following weighted Sobolev space
\[
\htilde = \left\{v\in\ltwor:||v||_{\ltwor} + ||\partial_r v||_{\ltwor} +
||\partial_\theta u||_{\ltworr} < \infty \right\}
\]
with inner product
\[
(u,v)_{\htilde}:=(u,v)_{\ltwor} + (\partial_r u,\partial_r v)_{\ltwor} +
(\partial_\theta u,\partial_\theta v)_{\ltworr}.
\]
The following lemma shows that $H^1(\Omega)$ and $\htilde$ are isometric.

\begin{lemma}
Let $\Phi:L^2(Q)\to\ltwor$ be defined by $\hat{u}\mapsto\hat{u}\circ F$.
Then the spaces $H^1(\Omega)$ and $\htilde$ are isometric via $\Phi$.
\end{lemma}
\begin{proof}
Let $\hat{u}\in H^1(\Omega)$ and, for $0<\rho<1$, let $B_\rho$ be the ball of
radius $\rho$ centred at the origin and $B_\rho^c$ its complement. For
$\Omega_\rho = \Omega\cap B_\rho^c$, the map $F:Q_\rho \to \Omega_\rho$ with
$Q_\rho := (\rho,1)\times(0,\Theta)$ is a bi-Lipschitz map, i.e. there exist
two constants $C_1,C_2 > 0$ such that
\[
C_1 \left|(r_1,\theta_1) - (r_2,\theta_2) \right| \leq
\left|F(r_1,\theta_1) - F(r_2,\theta_2) \right| \leq
C_2 \left| (r_1,\theta_1) - (r_2,\theta_2) \right|
\]
holds for all $(r_1,\theta_1),(r_2,\theta_2) \in Q_\rho$. Indeed, by the mean
value theorem we have
\[
\left| F(r_1,\theta_1) - F(r_2,\theta_2) \right| \leq
\norm{\nabla F} \left| (r_1,\theta_1) - (r_2,\theta_2) \right|
\]
and clearly $\norm{\nabla F}_\infty \leq 1$. In the same way, $F^{-1} :
\Omega_\rho \to Q_\rho$ is a smooth, bijective map and $\norm{F^{-1}}_\infty
\leq 1/\rho$; hence it is Lipschitz continuous and it follows that $F$ and
$F^{-1}$ are bi-Lipschitz when restricted to $Q_\rho$ and $\Omega_\rho$
respectively. As a result of \cite[Theorem 2.2.2.]{ziemer1989}, $u =
\Phi(\hat{u})$ is weakly differentiable on $Q_\rho$ and the chain rule holds.
For $n\in\NN$, define $u_n = \Phi|_{\Omega_{1/n}}(\hat{u})$ on $Q$ by
extending $\hat u$ by zero outside $\Omega_\rho$. For any $0<\rho<1$ one has
that
\begin{equation}
\label{eq:uniform_bound}
\norm{u}_{\tilde{H}^1(Q_\rho)} = \norm{\hat{u}}_{H^1(\Omega_\rho)}
\leq \norm{\hat{u}}_{H^1(\Omega)},
\end{equation}
so $u_n$ and its derivatives belong to the associated weighted Lebesgue
spaces. As a result of the monotone convergence theorem we have that
$u\in\htilde$ and
\[
\norm{u}_{\htilde} = \norm{\hat{u}}_{H^1(\Omega)}.
\]
Repeating the steps above, we can also show that for $u\in \htilde$ one has
$\Phi^{-1}(u)\in H^1(\Omega)$.
\end{proof}

It is then natural to define the following space in order to take into account
the boundary conditions
\[
\htildeo := \Phi(H^1_0(\Omega)).
\]

We are now ready to state the variational formulation of~\eqref{eq:poisson} in
both coordinate systems.

\begin{definition}[Weak form of the Poisson problem in Cartesian coordinates]

Find $\hat u\in H^1_0(\Omega)$ such that
\begin{equation}
\label{eq:poisson_weak}
\hat{a}(\hat{u}, \hat{v}) = \hat{b}(\hat{u}) \quad
\text{for all }\hat{v}\in H_0^1( \Omega),
\end{equation}
with
\[
\hat{a}(\hat{u},\hat{v}) =
\left( \hat{\nabla} \hat{u}, \hat{\nabla} \hat{v} \right)_{L^2( \Omega)},
\quad \hat{b}(\hat{u}) = \left(\hat{f}, \hat{v} \right)_{L^2( \Omega)}.
\]

\end{definition}

\begin{definition}[Weak form of the Poisson problem in polar coordinates]

Find $u\in\htildeo$ such that
\begin{equation}
\label{eq:poisson_polar}
a(u,v) = b(v) \quad \text{for all }v\in\htildeo,
\end{equation}
with
\[
\aligned
&a(u,v) = \int_0^1\int_0^{\Theta} \left(\partial_r u \partial_r v +
\frac{1}{r^2}\partial_\theta u \partial_\theta v \right)
r\,\mathrm{d}r\,\mathrm{d}\theta\\
&b(v) = \int_0^1\int_0^{\Theta} fv\, r\,\mathrm{d}r\,\mathrm{d}\theta.
\endaligned
\]

\end{definition}

It is well-known that~\eqref{eq:poisson_weak} is well posed and so
is~\eqref{eq:poisson_polar} thanks to the properties of the map $\Phi$ and of
the isometry shown above.

\section{The semi-discrete Poisson equation}
\label{se:semi-discrete}

The discretization of~\eqref{eq:poisson} with the scaled boundary finite
element method is based on a spatial semi-discretization that is described in
this section.

We introduce a partition of the parametrized boundary
$\theta\mapsto(\cos\theta,\sin\theta)$ given by
\[
\mathcal{T}_\Gamma = \left\{\theta_1,\dots,\theta_N\right\}
\]
and consider a finite dimensional approximation of $H^1(0,\Theta)$ generated
by a basis $\{e_i(\theta)\}_{i=1}^N$ with the property that
\[
e_i(\theta_j)=\delta_{ij}.
\]

\begin{remark}

The choice of $\{e_i(\theta)\}_{i=1}^N$ at this point is arbitrary. It could
be based on finite elements, splines, global Lagrange polynomials, etc.

Due to our choice of boundary conditions, we could also have defined the basis
$\{e_i(\theta)\}_{i=1}^N$ in $H^1_0(0,\Theta)$, but we prefer to avoid this in
order to allow our analysis to be extended more easily to more general boundary
conditions or to a situation where $\Theta=2\pi$.

\end{remark}

The main idea behind the semi-discretization is to consider families of
functions where the variables $r$ and $\theta$ are separated formally as
follows:
\[
u_s(r,\theta)=\sum_{i=1}^Nu_i(r)e_i(\theta).
\]
Ideally, we would like to have $u_i(r)=u(r,\theta_i)$ and this choice will be
used later in Subsection~\ref{se:numan} for the error analysis; it will
lead to the analogous of the interpolation operator for standard finite
elements.
In order to do so, we need to give sense to the radial trace $u(r,\theta_i)$.
For the sake of readability, we now introduce an abstract setting and we
postpone the actual definition of the involved functional spaces to
Subsection~\ref{se:numan}. Ultimately, we want to define a semi-discrete
space
\[
U^s := \left\{v_s\in\htilde:v_s = \sum_{i=1}^Nv_i(r)e_i(\theta)\text{ with }
v_i\in\tilde U \text{ for }1\leq i\leq N\right\},
\]
where $\tilde U$ is a suitable functional space on the interval $(0,1)$.
We will then consider its subspace $U^s_0=U^s\cap\htildeo$, so that the
semi-discretization of problem~\eqref{eq:poisson_polar} will read: find
$u_s\in U^s_0$ such that
\begin{equation}
\label{eq:poisson_polar_sd}
a(u_s,v_s) = b(v_s)\quad \text{for all }v_s\in U^s_0.
\end{equation}

We will prove (see Theorem~\ref{thm:completeness_vs}) that $U^s_0$ is a closed
subspace of $\htildeo$ so that problem~\eqref{eq:poisson_polar_sd} is uniquely
solvable and the error between $u$ and $u_s$ is bounded as usual by the best
approximation by C\'ea's lemma:
\begin{equation}
\|u-u_s\|_{\htilde}\le C\inf_{v\in U^s_0}\|u-v\|_{\htilde}.
\label{eq:cea}
\end{equation}

The solution of problem~\eqref{eq:poisson_polar_sd} is actually computed by
solving a system of ordinary differential equations where the unknowns are the
coefficients $u_i(r)$ of $u_s(r,\theta)$. This procedure is detailed in
Section~\ref{se:sbfem}.

In order to show the convergence of this procedure, we need to estimate the
right-hand side of~\eqref{eq:cea}.

\subsection{Error estimates for the interpolation operator}
\label{se:numan}

We plan to construct an interpolation operator
$\interp$ with values in $U^s$ that, if applied to smooth functions, would act
as follows
\[
(\interp u)(r,\theta) = \sum_{i=1}^N u(r,\theta_i) e_i(\theta).
\]

Since we will work with Sobolev functions, it is useful to define an adequate
trace-like operator that we are going to call the ``radial trace operator''.
To this end, the following bound is required.

\begin{lemma}
\label{lemma:trace_ineq}

For all $u\in C^\infty(\overline{Q})$ and $0 \leq \vartheta \leq \Theta$ we
have 
\begin{equation}
\label{eq:trace_ineq}
\int_0^1 ru^2(r,\vartheta)\mathrm{d}r \leq C \left( \norm{u}^2_{L^2_r(Q)} +
\norm{\partial_\theta u}^2_{L^2_r(Q)} \right)
\end{equation}
where $C>0$ only depends on $\Theta$.

\end{lemma}

\begin{proof}

Let $u\in C^\infty(\overline{Q})$ and assume, without loss of generality, that
$0 \leq \zeta < \Theta$. Then we have
\[
\int_\vartheta^\theta \partial_\zeta \left( u^2(r,\zeta) \right)
r\,\mathrm{d}\zeta = ru^2(r,\theta) - ru^2(r,\vartheta) \quad \text{for all }
\theta\in (\vartheta,\Theta].
\]
Reordering and integrating over $r$, we have
\[
\int_0^1 ru^2(r,\vartheta)\,\mathrm{d}r= \int_0^1 ru^2(r,\theta)\,\mathrm{d}r
-2\int_0^1 \int_\vartheta^\theta u(r,\zeta) \partial_\zeta u(r,\zeta)
r\,\mathrm{d}\zeta\,\mathrm{d}r.
\]
For the last term, we can apply H{\"o}lder's inequality, so that
\[
\aligned
- 2 \int_0^1 \int_\vartheta^\theta u(r,\zeta) \partial_\zeta u(r,\zeta)
r\,\mathrm{d}\zeta\,\mathrm{d}r  &\leq 2 \int_0^1 \int_0^{\Theta} \left\vert
u(r,\zeta) \partial_\zeta u(r,\zeta) \right\vert
r\,\mathrm{d}\zeta\,\mathrm{d}r
\\
& \leq 2\norm{u}_{L^2_r(Q)} \norm{\partial_\theta u}_{L^2_r(Q)} \\
& \leq \norm{u}^2_{L^2_r(Q)} + \norm{\partial_\theta u}^2_{L^2_r(Q)}.
\endaligned
\]
Finally, integrating over $\theta$, we have
\[
\Theta \int_0^1 ru^2(r,\vartheta)\,\mathrm{d}r \leq 2 \norm{u}^2_{L^2_r(Q)} +
\norm{\partial_\theta u}^2_{L^2_r(Q)},
\]
so that \eqref{eq:trace_ineq} holds for smooth functions.
\end{proof}

The following space on the interval $(0,1)$ will be used for the definition of
$\tilde U$
\[
H^1_r(0,1) = \left\{ u\in L^2_r(0,1) :
\int_0^1 \left( u'(r) \right)^2 r\,\mathrm{d}r < \infty \right\}.
\]
Given an angle $0\leq\vartheta\leq\Theta$, inequality ~\eqref{eq:trace_ineq}
shows that the natural norm for a space $U$ where the radial trace operator
can be defined, is
\[
\norm{u}_{U} := \left( \norm{u}^2_{\htilde} +
\norm{\partial_{r\theta} u}_{L^2_r( Q)}^2 \right)^{\frac{1}{2}}.
\]
It is apparent that not all functions in $C^\infty(\overline{Q})$ have
a bounded $U$-norm because in general $C^\infty(\overline{Q})$ is not included
in $\htilde$. This is due to the fact that
$\norm{\partial_\theta u}_{L^2_{1/r}( Q)}$ might not be bounded for some
$u\in C^\infty(\overline{Q})$. Hence, we define
\[
\tilde{\gamma}_\vartheta :
C^\infty(\overline{Q})\cap\htilde\to H^1_r(0,1)\qquad
u \mapsto u(\cdot,\vartheta)
\]
and extend it to the closure of $C^\infty(\overline{Q})\cap\htilde$ with
respect to the $U$-norm.
We denote by $U\subset\htilde$ this space and by $\gamma_\vartheta$ the
extension of the trace operator, so that we have a bounded radial trace
operator
\[
\gamma_\vartheta : U \to H_r^1(0,1)
\]
that extends the restriction operator $\tilde{\gamma}_\vartheta$ defined on
smooth enough functions.

It is then natural to choose $\tilde U=H^1_r(0,1)$, so that the definition of
$U^s$ reads as follows
\[
U^s := \left\{v_s\in\htilde:v_s = \sum_{i=1}^Nv_i(r)e_i(\theta)\text{ with }
v_i\in H^1_r(0,1) \text{ for }1\leq i\leq N\right\}.
\]

\begin{theorem}
\label{thm:completeness_vs}
The space of semi-discrete functions $U^s$ is a closed subspace of $\htilde$.
\end{theorem}

\begin{proof}
Let $(u_n)$ be a Cauchy sequence in $U^s$. By completeness, there is a
function $\tilde{u}$ such that $u_n\to \tilde{u}$ in $\htilde$. By Lemma
\ref{lemma:trace_ineq} we have
\[
\int_0^1 r\left|u_m(r,\theta_i) - u_n(r,\theta_i)\right|^2\,\mathrm{d}r\leq
C \norm{u_m - u_n}^2_{\htilde} \to 0
\]
as $n,m \to \infty$, so $(u_n(\cdot,\theta_i))$ is a Cauchy sequence in
$L_r^2(0,1)$ and by completeness there is a limit $u_n(\cdot,\theta_i) \to
u_i$ for each $i$.
Define $u = \sum_{i = 1}^N u_ie_i$ in $U^s$ and note that
\[
\aligned
\norm{u - \tilde{u}}_{L^2_r( Q)} &\leq \norm{u - u_n}_{L^2_r( Q)} +
\norm{\tilde{u} - u_n}_{L^2_r( Q)} \\
&\leq C \left( \sum_{i=1}^N \norm{u_i(\cdot) -
u_n(\cdot,\theta_i)}^2_{L^2_r(0,1)} \right)^{\frac{1}{2}} + \norm{\tilde{u} -
u_n}_{L^2_r( Q)} \to 0
\endaligned
\]
as $n\to \infty$, so $u = \tilde{u}$ and therefore $u_n\to u$ in $\htilde$.
\end{proof}

Given $u\in U$ we can then define the interpolant as
\[
(\interp u)(r,\theta) = \sum_{i=1}^{N} u_i(r) e_i(\theta),
\]
where $u_i(r)$ is defined as $\gamma_{\theta_i}(u)$ in $H^1_r(0,1)$. In order
to see that the interpolant is well defined, we need to show that $\interp u$
belongs to $\htilde$. To limit the technicalities, from now on in this section
we are assuming that $\{e_i\}$ is the basis of continuous piecewise linear
finite elements on $(0,\Theta)$. The general case can be handled with similar
arguments.

\begin{lemma}
For $u\in U$, we have $\Pi u \in \htilde$.
\end{lemma}

\begin{proof}
We have that $\Pi u\in\htilde$ if and only if
\begin{equation}
\label{eq:h1tilde_conditions}
\int_0^1\int_0^{2\pi} \left( r(\Pi u)^2 + r(\partial_r \Pi u)^2 +
\frac{1}{r}(\partial_\theta \Pi u)^2 \right)\,\mathrm{d}\theta\,\mathrm{d}r
< \infty.
\end{equation}
We apply Lemma~\ref{lemma:trace_ineq} and obtain
\begin{equation}
\label{eq:ineq_dr}
\aligned
\norm{\Pi u}^2_{L^2_r( Q)} + \norm{\partial_r \Pi u}^2_{L^2_r( Q)}
& \leq N \sum_{i=1}^N \Big( \int_0^1 ru_i^2(r)\,\mathrm{d}r
\int_0^{2\pi} e^2_i(\theta)\,\mathrm{d}\theta\\
&\quad+ \int_0^1 r \left(\partial_ru_i(r) \right)^2\,\mathrm{d}r
\int_0^{2\pi} e^2_i(\theta)\,\mathrm{d}\theta \Big) \\
&\leq CN\left(\norm{u}^2_{\htilde}+\norm{\partial_{r\theta}u}^2_{L^2_r(Q)}
\right)\sum_{i=1}^N\int_0^{2\pi}e^2_i(\theta)\,\mathrm{d}\theta\\
&\leq\frac{4\pi^2 Ch}{3h^2_{\min}}\left(\norm{u}^2_{\htilde}
+\norm{\partial_{r\theta}u}^2_{L^2_r( Q)}\right),
\endaligned
\end{equation}
where $h=\max_i(\theta_{i+1}-\theta_i)$ and 
$h_{\min}=\min_i(\theta_{i+1}-\theta_i)$.

For the third term in~\eqref{eq:h1tilde_conditions}, we fix $r\in(0,1)$ and
observe that
\[
u(r,\theta_i)=(\Pi u)(r,\theta_i),\quad u(r,\theta_{i+1})=
(\Pi u)(r,\theta_{i+1})
\]
for $i = 1,2,\dots,N-1$. Taking into account that $\partial_\theta \Pi u$ is
well defined in $(\theta_i, \theta_{i+1})$, we apply the mean value theorem
and 
\[
u(r,\theta_{i+1})-u(r,\theta_i)=\left(\theta_{i+1}-\theta_i \right)
\left(\partial_\theta \Pi u(r,\tilde{\theta})\right)
\]
holds for some $\tilde{\theta}\in(\theta_i,\theta_{i+1})$.
Since $\Pi u$ is linear in this interval, the following equality holds
\[
\left\lvert\partial_\theta\Pi u(r,\theta)\right\lvert^2
\left(\theta_{i+1}-\theta_i\right)^2=
\left\lvert\int_{\theta_i}^{\theta_{i+1}}\partial_\theta u(r,\zeta)
\,\mathrm{d}\zeta\right\lvert^2
\]
for all $\theta\in(\theta_i,\theta_{i+1})$ and $r\in(0,1)$.
After multiplying by $1/r$, integrating, and applying H\"older's inequality,
we have
\[
\int_0^1\int_{\theta_i}^{\theta_{i+1}}\frac{1}{r}
\left\lvert\partial_\theta \Pi u (r,\theta)\right\lvert^2
\,\mathrm{d}\theta\,\mathrm{d}r\le
\int_0^1\int_{\theta_i}^{\theta_{i+1}}\frac{1}{r}
\left\lvert\partial_\theta u(r,\theta)\right\lvert^2
\,\mathrm{d}\theta\,\mathrm{d}r.
\]
By integrating over each interval and summing up the terms, we have
\begin{equation}
\label{eq:ineq_dtheta}
\norm{\partial_\theta\Pi u}_{L^2_{1/r}(Q)}\le
\norm{\partial_\theta u}_{L^2_{1/r}(Q)}.
\end{equation}
Finally, putting~\eqref{eq:ineq_dr} and~\eqref{eq:ineq_dtheta} together, we
get
\[
\norm{\Pi u}_{\htilde}^2\le\max\left\{\frac{4\pi^2 Ch}{3h^2_{\min}},1\right\}
\left(2\norm{u}^2_{\htilde}+\norm{\partial_{r\theta}u}^2_{L^2_r(Q)}\right)
<\infty.
\]
\end{proof}

In the next theorem we prove the approximation properties of $U^s$. As usual,
we need to assume suitable regularity that will be characterized by the
following space
\[
U'=\left\{u\in U:\norm{\partial_{\theta\theta}u}_{L^2_{1/r}(Q)}<\infty\right\}.
\]

\begin{remark}
The space $U'$ requires extra regularity only in the angular variable
$\theta$. We will see in Section~\ref{se:num} that singular solutions (with
respect to the Cartesian coordinates) can be in $U'$ and be approximated
optimally by SBFEM.
\end{remark}

\begin{theorem}
Let $u$ be in $U'$. Then there exists $C>0$ independent of
$\mathcal{T}_\Gamma$ such that
\[
\norm{u-\Pi u}_{L^2_r(Q)}\le h^2\norm{\partial_{\theta\theta}u}_{L_{1/r}^2(Q)}
\]
and
\[
\norm{u-\Pi u}_{\htilde}\le Ch\left(\norm{\partial_{r\theta}u}^2_{L^2_r(Q)}
+\norm{\partial_{\theta\theta}u}^2_{L^2_{1/r}(Q)} \right).
\]
\label{th:conv}
\end{theorem}

\begin{proof}

Let $u\in C^\infty(\overline{Q})$ and define
$\varepsilon(r,\theta)=(u-\Pi u)(r,\theta)$.
Due to the properties of
the interpolation operator, we have $\varepsilon(r,\theta_i) = 0$ for $i =
1,...,N$. As a result, there is a $\tilde{\theta}_i\in(\theta_i,\theta_{i+1})$ such
that $\partial_\theta \varepsilon(r,\tilde{\theta}_i) = 0$. It follows that
\[
\partial_\theta\varepsilon(r,\vartheta)=
\int_{\tilde{\theta}_i}^\vartheta\partial_{\theta\theta}\varepsilon(r,\zeta)
\,\mathrm{d}\zeta=
\int_{\tilde{\theta}_i}^\vartheta\partial_{\theta\theta}u(r,\zeta)\,\mathrm{d}\zeta
\quad\text{for }\tilde{\theta}_i<\vartheta\le\theta_{i+1},
\]
since $\Pi u$ is linear in $(\theta_i,\theta_{i+1})$ in the $\theta$ direction.
Applying H\"older's inequality we have
\[
\left\lvert\partial_\theta\varepsilon(r,\vartheta)\right\rvert^2\le
\left(\theta_{i+1}-\theta_i\right)\int_{\theta_i}^{\theta_{i+1}}
\left\lvert\partial_{\theta\theta}u(r,\theta)\right\rvert^2\,\mathrm{d}\theta.
\]
Integrating over the domain and summing up the different terms corresponding
to each interval $(\theta_i,\theta_{i+1})$ we have
\begin{equation}
\label{eq:error_h1semi}
\norm{\partial_\theta\varepsilon}^2_{L_{1/r}(Q)}\le h^2
\norm{\partial_{\theta\theta}u}^2_{L_{1/r}(Q)}.
\end{equation}
Since both $||\cdot||_{L_{1/r}( Q)}$ and $\Pi$ are continuous, inequality
\eqref{eq:error_h1semi} can be shown to hold for all $u\in U^\prime$ by a
density argument. Likewise, for $\varepsilon(r,\vartheta)$ we have
\[
\left\lvert\varepsilon(r,\vartheta)\right\rvert^2=
\left\lvert\int_{\theta_i}^\vartheta\partial_\theta\varepsilon(r,\zeta)
\,\mathrm{d}\zeta\right\rvert^2\le h
\int_{\theta_i}^{\theta_{i+1}}
\left\lvert\partial_\theta\varepsilon(r,\zeta)\right\rvert^2\,\mathrm{d}\zeta.
\]
After integrating and using \eqref{eq:error_h1semi}, we have
\begin{equation}
\label{eq:error_l2r}
\norm{\varepsilon}_{L^2_r(Q)}^2\le
h^4\norm{\partial_{\theta\theta}u}^2_{L^2_{1/r}(Q)}.
\end{equation}

Finally, an estimate must be found for
$\norm{\partial_r\varepsilon}_{L^2_r(Q)}$.
Once again, we consider a smooth function $u$ and take into account that
$\partial_r\varepsilon(r,\theta_i)=0$ for all $r\in(0,1)$ and
$i=1,\dots,N$. Hence
\[
\partial_r\varepsilon(r,\vartheta)=
\int_{\theta_i}^\vartheta\partial_{r\theta}\varepsilon(r,\zeta)
\,\mathrm{d}\zeta
\]
and it follows that
\begin{equation}
\label{eq:error_h1rsemi_v1}
\norm{\partial_r \varepsilon}^2_{L_r^2(Q)}\le
h^2\norm{\partial_{r\theta}\varepsilon}^2_{L^2_{1/r}(Q)}.
\end{equation}
In the same way as~\eqref{eq:ineq_dtheta} is obtained, we apply the mean value
theorem to the function $\partial_r (\Pi u)(r,\vartheta)$ for
$\vartheta\in(\theta_i,\theta_{i+1})$ and obtain
\[
\left(\partial_{r\theta}(\Pi u)(r,\vartheta)\right)
\left(\theta_{i+1}-\theta_i\right)=
\int_{\theta_i}^{\theta_{i+1}}\partial_{r\theta}u(r,\zeta)\,\mathrm{d}\zeta
\]
and therefore we can establish that
\begin{equation}
\label{eq:ineq_dtheta_dr}
\norm{\partial_{r\theta}\Pi u}_{L^2_r( Q)}\le
\norm{\partial_{r\theta} u}_{L^2_r(Q)}.
\end{equation}
Given~\eqref{eq:ineq_dtheta_dr}, we have the following error estimate for all
smooth functions $u$:
\begin{equation}
\label{eq:error_h1rsemi}
\aligned
\norm{\partial_r\varepsilon}^2_{L_r^2( Q)}
&\le 2h^2\left(\norm{\partial_{r\theta} u}^2_{L^2_r(Q)}
+\norm{\partial_{r\theta}\Pi u}^2_{L^2_r( Q)}\right)\\
&\le 4h^2\norm{\partial_{r\theta} u}^2_{L^2_r( Q)} 
\endaligned
\end{equation}
which, by a density argument, holds for all $u\in U'$. Therefore,
putting~\eqref{eq:error_h1semi},~\eqref{eq:error_l2r},
and~\eqref{eq:error_h1rsemi} together, we obtain the required bound. More
precisely, we have
\[
\norm{u-\Pi u}^2_{\htilde}\le
h^2\left(h^2+4\right)\left(\norm{\partial_{r\theta}u}^2_{L^2_r( Q)}
+\norm{\partial_{\theta\theta}u}^2_{L_{1/r}(Q)}\right).
\]
\end{proof}

\section{Constructing semi-discrete solutions with SBFEM}
\label{se:sbfem}

In order to solve our problem, the scaled boundary finite element method
rewrites the formulation~\eqref{eq:poisson_polar_sd} as a system of ordinary
differential equations; this is carried out in Section \ref{subsec:ode}. This is possible thanks to the representation of semi-discrete functions in $U^s_0$ as the product of $r$-dependent functions $u_i(r)$ with $\theta$-dependent test functions. The resulting system of differential equations is supplemented with additional conditions arising from regularity requirements and the boundary conditions. This system can be solved by an analytical method under certain conditions, a process which is described in Section \ref{subsec:analytical}.

\subsection{Rewriting the Poisson equation as an ODE}
\label{subsec:ode}

Let us consider $u_s = \sum_{i=1}^Nu_ie_i$ and $v_s = \sum_{i=1}^N v_ie_i$.
An immediate consequence of the definition of the space of semi-discrete
functions $U_0^s$ is that the bilinear form $a:U_0^s\times U_0^s \to \RR$ and
the linear form $b:U_0^s\to\RR$ may be rewritten as follows:
\[
\aligned
a(u_s,v_s) &= \sum_{i,j=1}^N \int_0^1\int_0^{\Theta} \left(u^{\prime}_i(r)
v^{\prime}_j(r) e_i(\theta) e_j(\theta) + \frac{1}{r^2} u_i(r)v_j(r)
e^{\prime}_i(\theta) e^{\prime}_j(\theta) \right)
r\,\mathrm{d}r\,\mathrm{d}\theta
\\
& = \sum_{i,j=1}^N \bigg( \int_0^1 u^{\prime}_i(r) v^{\prime}_j(r)r\,
\mathrm{d}r \int_0^{\Theta} e_i(\theta) e_j(\theta)\,\mathrm{d}\theta
\\
&\qquad +\int_0^1 u_i(r)v_j(r) \frac{\mathrm{d}r}{r}  \int_0^{\Theta}
e^{\prime}_i(\theta) e^{\prime}_j(\theta)\, \mathrm{d}\theta \bigg) \\
& = \sum_{i,j=1}^n \left( \mathrm{A}_{ij} \int_0^1 u^{\prime}_i(r)
v^{\prime}_j(r)r\, \mathrm{d}r + \mathrm{B}_{ij} \int_0^1u_i(r)v_j(r)
\frac{\mathrm{d}r}{r}   \right)
\endaligned
\]
and
\[
\aligned
b(v_s) &= \sum_{j=1}^N  \int_0^1\int_0^{\Theta} f(r,\theta)
v_j(r)e_j(\theta)r\,\mathrm{d}r\,\mathrm{d}\theta \\
&= \sum_{j=1}^N  \int_0^1 F_j(r) v_j(r)r\,\mathrm{d}r,
\endaligned
\]
where
\[
\aligned
&\mathrm{A}_{ij} = \int_0^{\Theta} e_i(\theta)
e_j(\theta)\mathrm{d}\theta\\
&\mathrm{B}_{ij} = \int_0^{\Theta}
e^{\prime}_i(\theta) e^{\prime}_j(\theta) \mathrm{d}\theta\\
&F_j(r) = \int_0^{\Theta} f(r,\theta)e_j(\theta)\,\mathrm{d}\theta.
\endaligned
\]

We now proceed formally with the derivation of the differential equation. To
this aim, we will use the following integration by parts formula
\[
\int^1_0 u'_i(r)v'_j(r)r\,\mathrm{d}r = -\int^1_0 u''_i(r)v_j(r)r\,\mathrm{d}r
- \int_0^1 u'_i(r) v_j(r) \mathrm{d}r + u'_i(1)v_j(1)
\]
which is clearly valid for smooth enough $u_i$ and $v_i$.

In order to simplify our notation we introduce a name for the space of the
coefficients in $U^s$ as follows
\[
\boldsymbol{U}_0^s=\left\{\bu=(u_i)_{i=1}^N:u_i\in H_r^1(0,1)\text{ for }
1\le i\le N,\  \sum_{i=1}^N u_ie_i\in U_0^s\right\}.
\]
It follows that if $u_s\in U_0^s$ is smooth enough, it solves $a(u_s,v_s) = b(v_s)$ for all
$v_s\in U_0^s$ if and only if
\begin{equation}
\label{eq:poisson_sbfem_var}
\aligned
&\sum_{j=1}^N\int_0^1 v_j(r)\sum_{i=1}^N
\left[-r\mathrm{A}_{ij}u''_i(r)-\mathrm{A}_{ij}u'_i(r)
+\frac{1}{r}\mathrm{B}_{ij}u_i(r)-F_j(r)r\right]\mathrm{d}r=0
\endaligned
\end{equation}
holds for all $\bv=(v_i)_{i=1}^N\in\boldsymbol{U}_0^s$. Moreover, as a result of the fundamental lemma of calculus of variations,~\eqref{eq:poisson_sbfem_var} holds if and only if
\begin{equation}
\label{eq:poisson_sbfem_ODE}
\sum_{j=1}^N\left(r\mathrm{A}_{ij}u''_j(r)+\mathrm{A}_{ij}u'_j(r) -\frac{1}{r}\mathrm{B}_{ij}u_j(r)\right)=rF_i(r)\,\text{for a.e. }r\in(0,1)
\end{equation}
is satisfied for all $i=2,\dots,N-1$, since $v_1 = v_N = 0$ by $\bv\in \boldsymbol{U}_0^s$. Because we seek a solution $\bu\in \boldsymbol{U}_0^s$, we must also enforce the boundary conditions $u_1 = u_N = 0$ and $u_i(1) = 0$ for $i = 2,\dots,N-1$. Furthermore, we impose the compatibility condition $u_i(0) = u_j(0)$ for all $i,j = 1,\dots,N$ to avoid a singularity at the center of the disk. In his, we guarantee that the function $u = \sum_{i = 1}^Nu_ie_i$ belongs to $\htildeo$.

Denote by $\mA$ and $\mB$ the matrices in $\RR^{N\times N}$ with components $A_{ij}$ and $B_{ij}$ and by $\bF(r)$ the vector function with components $F_i(r)$. We modify the columns and rows in $\mA$ and $\mB$ and we set $F_i(r) = 0$ for $i = 1$ and $N$ in order to enforce $u_1 = u_N = 0$. Then, a semi-discrete solution of the Poisson equation $\bu=(u_i)_{i=1}^N$ will satisfy the following conditions:
\begin{equation}
\label{eq:poisson_sbfem_ODE_complete}
\left\{
\aligned
&r^2\mA \bu'' + r\mA\bu' - \mB\bu = r^2\bF&&\text{for }r\in(0,1),\\
&\bu(1) = 0, \\
&u_i(0) = u_j(0) && \text{for all $i,j = 1,\dots,N$}.
\endaligned
\right.
\end{equation}
SBFEM provides a methodology for the construction of solutions for \eqref{eq:poisson_sbfem_ODE_complete}, see the next section. Whenever such solutions exist and satisfy the integration by parts formula for all $\bv\in \boldsymbol{U}_0^s$, then they correspond with the unique solution in $U^s_0$.

%

\subsection{Solving the ODE analytically}\label{subsec:analytical}

Following \cite{wolf2003}, a solution to \eqref{eq:poisson_sbfem_ODE_complete} can be constructed by finding a family of functions that satisfy the homogeneous ODE together with a particular
solution. In order to construct the homogeneous solution, we shall define the matrix $\mE\in\RR^{2N\times 2N}$ by
\[
\mE=\left(
\begin{matrix}
0 &\mA^{-1}\\
\mB & 0 
\end{matrix}
\right).
\]
Note that $\mE$ arises when the ODE in \eqref{eq:poisson_sbfem_ODE_complete} is rewritten as a first order differential equation by introducing the additional variable $\bq(r)=r\mA\bu'(r)$. If $(\lambda, (\bphi,\bpsi)^T)$ is an eigenpair of $\mE$, with $\bphi,\bpsi\in\CC^N$, we can see that
\begin{equation}
\label{eq:eigenvalues}
	\mB\bphi = \lambda^2 \mA \bphi.
\end{equation}
Furthermore, for $\bu(r) = r^{\lambda}\bphi$ we have that
\begin{equation}
\label{eq:homogeneous_ODE}
	r^2\mA\bu'' + r\mA\bu' - \mB\bu = 0
\end{equation}
for all $r>0$. That is, $\bu(r) = r^{\lambda}\bphi$ is a homogeneous solution of the ODE in \eqref{eq:poisson_sbfem_ODE_complete}. This idea can be extended in such a way that we get homogeneous solutions as a linear combination of $N$ linearly independent functions that satisfy \eqref{eq:homogeneous_ODE}. Indeed, first note that $\mE$ is a Hamiltonian matrix and therefore, for every eigenvalue $\lambda\in\CC$, we will also have that $-\lambda$, $\overline{\lambda}$ and $-\overline{\lambda}$ are eigenvalues of $\mE$. Consider  the following subset of $N$ eigenpairs of $\mE$:
\[
	\left\lbrace (\lambda_i, \left( \begin{array}{c} \bphi_i \\
	\bpsi_i
	\end{array} \right) ) : \R{(\lambda_i}) \geq 0,\quad \bphi_i,\bpsi_i\in\CC^N,\quad i =1,\dots,N \right\rbrace .
\]
From equation \eqref{eq:eigenvalues}, we see that the pair $(\lambda_i^2,\bphi_i)$ is an eigenpair of the one dimensional Laplace problem discretised with piecewise polynomials. As a result, $\lambda_i\in\RR$ and $\bphi_i\in\RR^N$. Moreover, the vectors $(\bphi_i)$ form a basis of $\RR^N$. 
Then, for any $N$-tuple of real numbers $c_1,\dots,c_N$, we have that
\[
	\bu_h(r) = \sum_{i = 1}^N c_i r^{\lambda_i} \bphi_i
\]
is a homogeneous solution to the ODE in \eqref{eq:poisson_sbfem_ODE_complete}. Given a particular solution $\bu_p$ for which 
\[
	r^2\mA\bu_p'' + r\mA\bu_p' - \mB\bu_p = r^2\bF \quad \text{on $(0,1)$},
\]
the constants $c_1,\dots c_N$ are then set by enforcing the boundary conditions $\bu_h(1) + \bu_p(1) = 0$. The resulting function $\bu = \bu_h + \bu_p$ is a solution of \eqref{eq:poisson_sbfem_ODE_complete}. 

\begin{remark}
	In the construction of $\bu$ we only consider non-negative eigenvalues because we require $u_i(r)\in H^1_r(0,1)$ for all $i = 1,\dots,N$. Moreover the condition $u_i(0) = u_j(0)$ for all $i,j = 1,\dots,N$ holds because if $\lambda_i = 0$ then $\bphi$ is in the kernel of $\mB$, which consists of vectors whose entries are all equal.
\end{remark}

The analytical construction of a particular solution relies on the form of $\bF(r)$. For example, if $\bF(r) = r^\alpha \boldsymbol{f}$, where $\boldsymbol{f}\in\RR^N$ and $\alpha\in\RR$, then it is simple to see that $\bu_p(r) = r^{\alpha + 2}\bphi_p$, with $\bphi_p\in\RR^N$ solving the linear system
\[	
	\left( (\alpha + 2)^2\mA - \mB \right) \bphi_p = \boldsymbol{f},
\]
is a particular solution of \eqref{eq:poisson_sbfem_ODE}.

\section{Numerical examples}
\label{se:num}

One of the main features of the method presented in the previous sections is
that it can achieve high order of convergence also in the presence of singular
solutions. We are going to show this behavior with two simple examples. The approximate solutions to \eqref{eq:poisson_sbfem_ODE} are obtained by means of the analytical method described in Section \ref{subsec:analytical}. In these tests, the basis functions $e_i$ in the angular direction are piecewise polynomials of order 1 or higher.

\subsection{A first numerical test}

We take $\Theta = 3\pi/2$ and consider the function
\[
u_e (r,\theta) = r^{\frac{2}{3}}\sin{\left(\frac{2}{3}\theta\right)}.
\]
This function satisfies $\Delta \hat{u}_e = 0$ on $\Omega$, where $\hat{u}_e = \Phi^{-1}(u_e)$, so it is a solution to
the homogeneous problem: find $\hat{u}\in H^1(\Omega)$ such that
\[
\aligned
-\Delta \hat{u} &= 0&&\text{in $\Omega$}\\
\hat{u} &= \hat{u}_e&&\text{on $\partial\Omega$}.
\endaligned
\]
Our theory can be easily extended to accommodate non-homogeneous boundary
conditions.
Moreover, it is well-known that $\hat{u}_e$ belongs to $H^1(\Omega)$ but not to
$H^2(\Omega)$. Indeed, this follows from the fact that the following
inequality holds for any $0<R<1$ and $0<\varepsilon<R$
\[
\aligned
\norm{\hat{u}_e}_{2,\Omega}^2 &\ge
\int_\varepsilon^R\int_0^{\frac{3}{2}\pi}
\left(\frac{\partial u_e}{\partial r}\right)^2r\,\mathrm{d}\theta\,\mathrm{d}r
=\left(\frac{4}{81}\int_0^{\frac{3}{2}\pi}\mathrm{sin}^2
\left(\frac{2}{3}\theta \right)\,\mathrm{d}\theta\right)
\int_\varepsilon^R r^{-\frac{5}{3}}\,\mathrm{d}r\\ 
&= C \left( \varepsilon^{-\frac{2}{3}} - R^{-\frac{2}{3}}\right)
\endaligned
\]
which tends to infinity as $\varepsilon$ goes to $0$. On the other hand, it can be easily seen that $u_e$ belongs to $U'$ and this
makes it possible to use the result of Theorem~\ref{th:conv} which implies, in
particular, that first order elements achieve second order of convergence in
$L^2$ even in presence of a corner singularity.

We report in Figure~\ref{fg:numtest} the results of our numerical test that
confirm our theoretical findings. We include also higher order approximations
(up to order six) and a convergence plot in $H^1$.


\begin{figure}
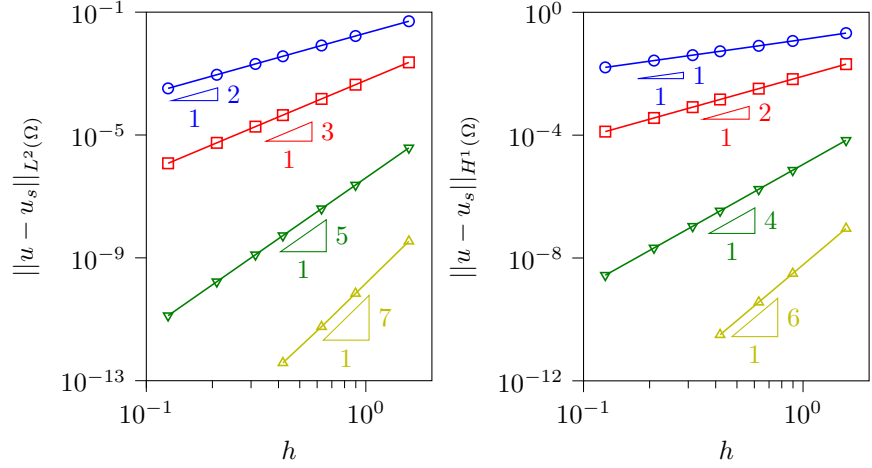

\includegraphics[width=0.45\linewidth,height=.5\linewidth]{L2_simple.tikz}
\includegraphics[width=0.45\linewidth,height=.5\linewidth]{H1_simple.tikz}
\caption{Convergence plots of $L^2$ and $H^1$-errors for the first numerical test with basis functions $e_i$ of polynomial order 1, 2, 4, and 6.}
\label{fg:numtest}
\end{figure}

\subsection{A second numerical test}

For the second numerical test, we consider a slightly more complicated example. Once again, we set $\Theta = 3\pi/2$ and consider the function
\[
	v_e(r,\theta) = r^{\frac{2}{3}}\left( \left(1 - \frac{4\theta}{3\pi} \right) \cos{(2\theta/3)} - \frac{4}{3\pi}\log{(r)}\sin{(2\theta/3)} \right).
\]
For this function we once again have that $\Delta \hat{v}_e = 0$, where $\hat{v}_e = \Phi^{-1}(v_e)$, and $v_e = r^{2/3}$ at $\theta = 0$ and $\theta = \Theta$. This function therefore is a solution to the problem: find $\hat{u}\in H^1(\Omega)$ such that
\[
\aligned
-\Delta \hat{u} &= 0 &&\text{in $\Omega$}\\
\hat{u} &= \hat{v}_e&&\text{on $\partial\Omega$}.
\endaligned
\]

The results to our numerical tests are presented in Figure \ref{fg:numtest2} for polynomial orders 1 and 2. The resulting convergence rates confirm the theoretical predictions from Section \ref{se:semi-discrete}.

\begin{figure}
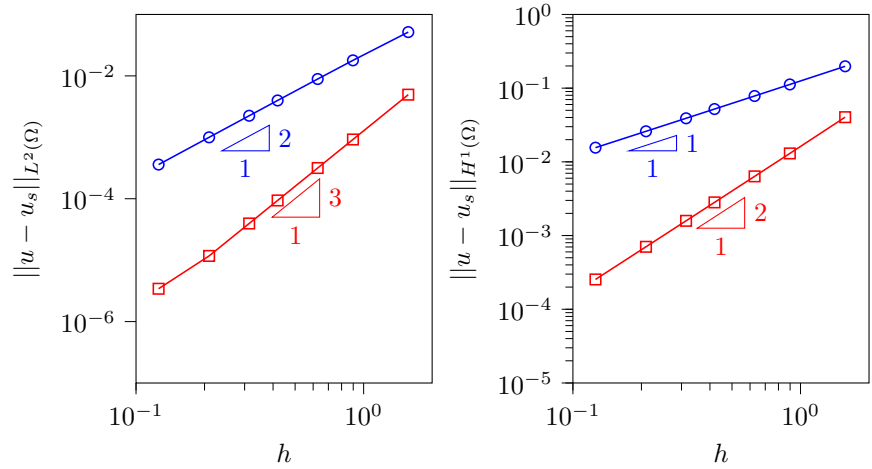

\includegraphics[width=0.45\linewidth,height=.5\linewidth]{L2_complex.tikz}
\includegraphics[width=0.45\linewidth,height=.5\linewidth]{H1_complex.tikz}
\caption{Convergence plots of $L^2$ and $H^1$-errors for the second numerical test with basis functions $e_i$ of polynomial order 1 and 2.}
\label{fg:numtest2}
\end{figure}

\bibliographystyle{amsplain}
\bibliography{ref}

\end{document}